\documentclass[11pt, reqno]{amsart}
\usepackage{color}
\usepackage{comment}
\usepackage[]{hyperref}
\usepackage{amsmath}
\usepackage{amsfonts}
\usepackage{amssymb}
\usepackage{amscd}
\usepackage{tikz}
	\usetikzlibrary{arrows,snakes}
	\usetikzlibrary{decorations.pathreplacing} 
	\usetikzlibrary{matrix}
\usepackage{stmaryrd}
\usepackage{multirow}
\usepackage{mathtools}
\usepackage{subfigure}
\usepackage{verbatim}
\usepackage{pinlabel}
\usepackage[all,cmtip,knot]{xy}
\usepackage{enumerate, enumitem, comment}
\usepackage[normalem]{ulem}

\newtheorem{Theorem}{Theorem}

\newtheorem{proposition}[Theorem]{Proposition}
\newtheorem*{namedtheorem}{\theoremname}
\newcommand{\theoremname}{testing}

\theoremstyle{definition}

\theoremstyle{remark}
\newtheorem{remark}[Theorem]{Remark}

\def\ref{\textup{ref}}

\def\HFred{HF_{red}}

\author[Jennifer Hom]{Jennifer Hom}
\thanks{The first author was partially supported by NSF grants DMS-1128155, DMS-1307879, DMS-1552285, and a Sloan Research Fellowship.}
\address {School of Mathematics, Georgia Institute of Technology, Atlanta, GA 30332}
\address{School of Mathematics, Institute for Advanced Study, Princeton, NJ 08540}
\email{hom@math.gatech.edu}

\author[Tye Lidman]{Tye Lidman}
\thanks{The second author was supported by NSF grant DMS-1128155.}
\address {School of Mathematics, Institute for Advanced Study, Princeton, NJ 08540}
\email {tlid@math.utexas.edu}

\title{A note on surgery obstructions and hyperbolic integer homology spheres}

\begin{document}
\maketitle

\begin{abstract}
Auckly gave two examples of irreducible integer homology spheres (one toroidal and one hyperbolic) which are not surgery on a knot in the three-sphere. Using Heegaard Floer homology, the authors and Karakurt provided infinitely many small Seifert fibered examples. In this note, we extend those results to give infinitely many hyperbolic examples, as well as infinitely many examples with arbitrary JSJ decomposition.
\end{abstract}

\section*{}

Lickorish \cite{Lickorish} and Wallace \cite{Wallace} proved that any closed, oriented three-manifold can be obtained by surgery on a link in the three-sphere. Thus, a natural question to ask is which manifolds can be described via the simplest possible surgery description, i.e., as surgery on a knot. Irreducible integer homology spheres are a particularly interesting family to consider, since the simplest obstructions (e.g., \cite{BL}) to being surgery on a knot all vanish. Note that Gordon and Luecke \cite{GL} showed that a reducible integer homology sphere can never be surgery on a knot. Auckly \cite{Auckly} provided the first two examples (one toroidal and one hyperbolic) of irreducible integer homology spheres which are not surgery on a knot, answering \cite[Problem 3.6(C)]{K} in the affirmative. For over 15 years, these two manifolds were the only known examples, until the authors and Karakurt \cite{HKL} provided an infinite family of small Seifert fibered integer homology spheres which are not surgery on a knot. In this note, we refine that result to give an infinite family of hyperbolic examples as well.

\begin{Theorem}\label{thm:main} 
There exist infinitely many hyperbolic integer homology spheres which are not surgery on a knot in $S^3$.  Similarly, one can construct infinitely many examples with arbitrarily complicated JSJ decomposition.  Finally, one can arrange that none of these examples are rationally homology cobordant.
\end{Theorem}

Theorem~\ref{thm:main} is a consequence of the following results about the behavior of Heegaard Floer homology under Dehn surgery. The following results were originally proved for knots in $S^3$, but hold more generally for knots in arbitrary integer homology sphere L-spaces.

\begin{proposition}[{\cite[Theorem 3]{Gainullin}}]\label{prop:genusred}
Let $Y$ be an integer homology sphere L-space and $K \subset Y$ a genus one knot. Then $U^2 \cdot \HFred(Y_{1/n}(K)) = 0$ for any $n$.
\end{proposition}

\begin{proposition}\label{prop:hfsurgery}
Let $Y$ be an integer homology sphere L-space and $K \subset Y$ a genus one knot.  Then $|d(Y) - d(Y_{1/n}(K))| \leq 2$ for any $n$.
\end{proposition}

\begin{proof}
By  \cite[Proposition 1.6]{NiWu}, we have that
\begin{equation}\label{eq:d-surgery}
d(Y_{1/n}(K)) - d(Y) = -2V_0 \text{ for } n > 0,
\end{equation}
where $V_0$ is the integer-valued knot invariant defined in \cite[Definition 7.1]{RasmussenThesis} (see also \cite[Section 2.2]{NiWu}).

Since being an integer homology sphere L-space is preserved under orientation reversal, by \eqref{eq:d-surgery}, it suffices to prove that $V_0 \leq 1$ for any genus one knot in an integer homology sphere L-space.  This is well-known for knots in $S^3$ (see, for instance, \cite[Theorem 2.3]{RasmussenGT}) and the same arguments apply in general.    
\end{proof}

\begin{proof}[Proof of Theorem~\ref{thm:main}]
Let $Y_j = \#_j \Sigma(2,3,5)$ for $j \geq 6$, where $\Sigma(2,3,5)$ is oriented as the boundary of the positive-definite $E_8$ plumbing. Recall that $Y_j$ is an L-space and that $d(Y_j) = -2j$.  Consider a genus one knot $K_j$ in $Y_j$ and let $Y_{j,n}$ be the manifold obtained by $1/n$-surgery on $K_j$.  By Proposition~\ref{prop:hfsurgery}, $d(Y_{j,n}) \leq -10$ and by Proposition \ref{prop:genusred}, $U^2 \cdot \HFred(Y_{j,n}) = 0$.  The same arguments used to prove Theorem 1.2 of \cite{HKL} imply that if $Y_{j,n}$ is surgery on a knot in $S^3$ with $d(Y_{j,n}) \leq -10$, then $U^2 \cdot \HFred(Y_{j,n}) \neq 0$.  
Therefore, we obtain a contradiction.  
  
It is now straightforward to construct $Y_{j,n}$ of the desired forms.  Indeed, by \cite[Proposition 5.4]{Tsutsumi}, there exists a hyperbolic knot $K_j$ in $Y_j$ which is genus one.  By Thurston's hyperbolic Dehn surgery theorem, $Y_{j,n}$  will be hyperbolic for all but finitely many $n$.  For a more complicated JSJ decomposition, simply take the knot $K_j$ and perform any number of cables and/or Whitehead doubles of this knot, making sure that the last step is a Whitehead double.  The resulting knot will be genus one, and the number of hyperbolic pieces (respectively Seifert pieces) of the knot exterior will be determined by the number of Whitehead doubling operations (respectively cabling operations).  Again, $Y_{j,n}$ will have the same JSJ decomposition as the satellite knot exterior for all but finitely many $n$.   

Finally, by taking $j \to \infty$, we can guarantee that the different $Y_{j,n}$ have different $d$-invariants, and thus live in different rational homology cobordism classes.   
\end{proof}

\begin{remark}
A similar argument can be used to show that for any integer homology sphere, there are infinitely many hyperbolic and irreducible toroidal homology spheres that can not be obtained by surgery on a knot in $Y$.  This comes from combining Propositions \ref{prop:genusred} and \ref{prop:hfsurgery} with the arguments used in the proof of \cite[Theorem 1.7]{HKL}.
\end{remark}

\section*{Acknowledgments} 
The results were obtained while the authors were members at the Institute for Advanced Study.  We would like to thank the IAS for its hospitality. We also thank Steve Boyer and Alan Reid for helpful conversations; we are particularly grateful to Alan for informing us of \cite[Proposition 5.4]{Tsutsumi}.

\bibliographystyle{amsalpha}

\bibliography{references}

\end{document}